\documentclass[a4paper,11pt]{amsart}
\usepackage{amsfonts,amsmath,amssymb,latexsym}
\usepackage{mathrsfs}
\usepackage[utf8x]{inputenc}

\usepackage[usenames,dvipsnames]{xcolor}

\usepackage{fullpage}
\usepackage{enumerate}
\usepackage{tikz}

\def\data{\ifcase\month\or January\or February \or March\or April\or May
\or June\or July\or August\or September\or October\or November
\or December\fi\space\number\day, \number\year}

\def\pv#1{\ensuremath{{\mathsf{#1}}}}

\def\GAP{\textsf{GAP}}

\def\NN{\hbox{\mbox{$\mathbb N$}}}
\def\N{\hbox{\mbox{$\mathbb N$}}}

\newtheorem{theorem}{Theorem}[section]
\newtheorem{proposition}[theorem]{Proposition}
\newtheorem{lemma}[theorem]{Lemma}
\newtheorem{corollary}[theorem]{Corollary}
\newtheorem{fact}[theorem]{Fact}
\newtheorem{question}[theorem]{Question}
\newtheorem{conjecture}[theorem]{Conjecture}
\theoremstyle{remark}
\newtheorem{remark}[theorem]{Remark}

\newtheorem{example}[theorem]{Example}

\def\max{\mathop{max}}
\def\rqns{{\tt rqns}}

\def\all{^>}

\title{Rees quotients of numerical semigroups} \author{Manuel Delgado}
\address{Manuel Delgado, Centro de Matemática da Universidade do Porto, Departamento de Matematica, Faculdade de Ci\^encias,
  Universidade do Porto, Rua do Campo Alegre 687, 4169-007 Porto, Portugal}
\email[M. Delgado]{mdelgado@fc.up.pt} \thanks{The first author was partially funded by the European Regional Development Fund
through the program COMPETE and by the Portuguese Government through the FCT -
Funda{\c c}{\~a}o para a Ci{\^e}ncia e a Tecnologia under the project
PEst-C/MAT/UI0144/2011. He benefited also of the sabbatical grant SFRH/BSAB/1156/2011.}

\author{V\'\i tor H. Fernandes} \address{{\sc V\'\i tor H. Fernandes},
Departamento de Matem\'atica,
Faculdade de Ci\^encias e Tecnologia,
Universidade Nova de Lisboa,
Monte da Caparica,
2829-516 Caparica,
Portugal;
also (second address):
Centro de \'Algebra da Universidade de Lisboa,
Av. Prof. Gama Pinto 2,
1649-003 Lisboa,
Portugal;
}
\email[V. H. Fernandes]{vhf@fct.unl.pt} \thanks{The second author gratefully acknowledges support of FCT and PIDDAC, within the projects ISFL-1-143 and
PTDC/MAT/69514/2006 of CAUL}

\begin{document}
\begin{abstract}
  We introduce a class of finite semigroups obtained by considering Rees
  quotients of numerical semigroups.
  Several natural questions concerning this class, as well as particular
  subclasses obtained by considering some special ideals, are answered while
  others remain open.  We exhibit nice presentations for these semigroups and
  prove that the Rees quotients by ideals of $\NN$, the positive integers
  under addition, constitute a set of generators for the pseudovariety of
  commutative and nilpotent semigroups.
\end{abstract}

\maketitle

{\tiny \data}

\section{Introduction and motivation}\label{sec:motivation}
A \emph{numerical semigroup} is a co-finite subsemigroup of the non-negative
integers, under addition.
It is well known that a numerical semigroup has a (unique) minimal set of
generators, which is finite. The smallest integer from which all the integers
belong to a numerical semigroup is called the \emph{conductor} of that
semigroup. The (finite) set of elements of the semigroup not greater than the
conductor, named \emph{small elements}, also determines the semigroup. We
just mentioned two (in general) different finite sets of integers that
determine a given numerical semigroup but others could be considered. This
motivates the following somehow vague question, which has been the starting
point for the research presented in this paper:

\begin{question}\label{quest:original}
  Can a numerical semigroup be thought as a finite semigroup?
\end{question}

We consider this question imprecise due to the usual fact that one does not
distinguish between isomorphic semigroups. In the above motivation, concrete
finite sets determining the numerical semigroups have been considered.

The following is a particular case of the approach we have taken. Given a
numerical semigroup $S$ we choose a cutting point $k\in S$ and a
distinguished element $\infty$ not in $S$. Let $F=\{s\in S\mid s\le k\}$ and
take $Q=F\cup \{\infty\}$. In $Q$ we define a (commutative and associative)
operation as follows: $\infty\oplus q=q\oplus\infty=\infty$, for any $q\in
Q$; for $a,b\in F$, $a\oplus b= a+b$ if $a+b\in F$ and $a\oplus b=\infty$
otherwise.  The set $F$ is called the \emph{finite part} of the finite
commutative semigroup $Q$. Notice that $Q$ is a Rees quotient of $S$ by an
ideal (formed by the elements of $S$ that are greater than $k$), which
justifies the use of the terminology \emph{Rees quotient semigroup}.

One could be tempted to ask whether taking as cutting point an integer
greater than any of the minimal generators, or greater than the conductor,
the finite quotient obtained by the above construction contains sufficient
data to determine the numerical semigroup and somehow answering positively
Question~\ref{quest:original}. This is far from being the case: we can have
more than one numerical semigroup giving rise to the same quotient (up to
isomorphism), as many of the forthcoming examples show.

This paper is written as follows: after this brief section mainly devoted to
the motivation for the research presented, we give the main definitions and
introduce the notation to be used.  We proceed with a section containing
examples and simple remarks that intend to give answers to several questions
that can naturally be raised once one wants to consider the class of finite
semigroups introduced in this paper: Rees quotients of numerical semigroups.
Next we obtain a nice presentation for a Rees quotient numerical semigroup
once it is known the defining pair (numerical semigroup, ideal). Observe that
the study of presentations appears naturally in both the theories of
numerical semigroups and of finite semigroups.
In particular, classifying finite semigroups into pseudovarieties is one of
the main subjects of study in finite semigroup theory. In
Section~\ref{sec:pseudovarieties} we determine the pseudovariety generated by
the Rees quotients of numerical semigroups (in fact, the quotients of $\NN$
suffice): it is the class of finite commutative nilpotent semigroups.
In a final section we raise several questions for which we have not been able
to get answers and left them as open problems, pointing out that some
research work on this subject can be pursued.

\section{Definitions and notation}\label{sec:defs_and_notation}
Our reference for numerical semigroups is the book by Rosales and
García-Sánchez~\cite{RosalesGarcia2009Book-Numerical}.  One may use the
\GAP~\cite{GAP4} package~\cite{DelgadoMoraisGarcia-Sanchez:numericalsgps} for
computations with numerical semigroups so as with (relative) ideals. Other
relevant reference for this paper is \cite{Almeida1995Book-Finite}. It
contains everything we need on finite semigroups, namely the result on
pseudovarieties of commutative semigroups used to prove
Theorem~\ref{th:pseudovariety_generators}.

Except for numerical semigroups or their quotients, as usual, we always
assume multiplicative notation.
\subsection{Numerical Semigroups}\label{subsec:numerical_semigroups}
In this paper, for convenience, we shall use the terminology \emph{numerical semigroup} for a
co-finite subsemigroup of the non-negative integers, under addition.
Traditionally, it is required that a numerical semigroup contains the $0$,
which works as an identity, and therefore numerical semigroups are
monoids. The terminology ``numerical monoid'' also appears in the
literature. This makes usually no difference in the theory development. 
Despite, the following notation is useful: 
given a numerical semigroup $S$, we denote by $S^0$ the monoid obtained from $S$ by
adjoining the integer $0$ (which is the identity).  In particular, denoting
by $\NN$ the semigroup of positive integers, $\NN^0$ denotes the monoid of
non-negative integers under addition.

Given a subsemigroup $S$ of $\NN$, let $d=\gcd(S)$. It is easy to see that
$\{\frac{s}{d}\mid s\in S\}$ is a co-finite subsemigroup of $\NN$ (thus a
numerical semigroup) that is isomorphic to $S$. Therefore, each isomorphism
class of the set of subsemigroups of $\NN$ contains a numerical
semigroup. That contains exactly one (which implies that different numerical
semigroups are non isomorphic) is well known and is a consequence of the
following proposition which may be seen as a direct proof of this fact and is
stated here for the sake of completeness.

\begin{proposition}\label{prop:isomorphism}
  Let $\varphi:S\to T$ be a surjective homomorphism from a numerical semigroup $S$
  to a numerical semigroup $T$. Then $S=T$.
\end{proposition}
\begin{proof}
  Let $A=\{a_1,a_2,\ldots,a_n\}$ be a set of generators of $S$. In
  particular, $\gcd (A)=1$. Since $\varphi$ is surjective, then $\varphi(A)$ is a
  set of generators of $T$ and therefore, as $T$ is a numerical semigroup, we
  obtain $\gcd (\varphi(A))=1$.
 
  For each $i\in \{1,\ldots,n\}$, we have
  $a_1\varphi(a_i)=\varphi(a_1a_i)=\varphi(a_ia_1)=a_i\varphi(a_1)$. It
  follows that $a_1 \mid a_i\varphi(a_1)$, for all $i\in
  \{1,\ldots,n\}$. Since $\gcd (a_1,a_2,\ldots,a_n)=1$, we get that $a_1 \mid
  \varphi(a_1)$ and that $\frac{\varphi(a_1)}{a_1}\mid \varphi(a_i)$, for any
  $i\in \{1,\ldots,n\}$. This implies that $\frac{\varphi(a_1)}{a_1}\mid \gcd
  (\varphi(A))$ and so $a_1=\varphi(a_1)$. Consequently $a_i=\varphi(a_i)$,
  for any $i\in \{1,\ldots,n\}$. Thus $S=T$, as required.
\end{proof}
\subsection{Notable elements}\label{subsec:notable}
Let $S$ be a numerical semigroup.  The greatest integer not belonging to $S$
is called the \emph{Frobenius number} of $S$ and denoted $F(S)$. The
successor of the Frobenius number is called the \emph{conductor} of $S$. It is
the least element of $S$ such that all the integers greater than it belong to
$S$ and is denoted by $c(S)$.  The least (positive) element of $S$, which is
also the minimum of the unique minimal set of generators, is called the
\emph{multiplicity} of $S$ and is denoted by $m(S)$.  The \emph{embedding
  dimension} of $S$ is the cardinality of the minimal generating set of $S$
and is denoted by $e(S)$.

The positive integers that do not belong to $S$ are called the \emph{gaps} of
$S$ and the number of such elements is called the \emph{genus} of $S$ and
denoted $g(S)$.
\subsection{Ideals}
When studying numerical semigroups it is common to consider relative ideals
(see~\cite{BarucciDobbsFontana1997Book-Maximality}).  Given a numerical
semigroup $S$, a relative ideal $I_S$ of $S$ has a (unique) finite minimal
\emph{ideal generating system}, say $G$, such that $I_S=G+S^0$.  Since we are
interested in forming quotients, we are only interested in those relative
ideals that are contained in their ambient semigroups (and which are in fact
the semigroup ideals). As in~\cite{RosalesGarcia2009Book-Numerical}, these
are called \emph{ideals}. Notice that the ideals of a numerical semigroup $S$
are precisely those relative ideals whose minimal ideal generating system is
contained in $S$. Given a positive integer $k$, the set $I_k(S)=\{x\in S\mid
x\ge k\}$ is an ideal of $S$; we say that it is an \emph{ideal determined by
  the cutting point $k$}. When the ambient semigroup is understood, we write
$I$ (respectively $I_k)$ for $I_S$ (respectively $I_k(S)$).

Let $S$ be a numerical semigroup with multiplicity $m$ and conductor $c$. 
Observe that if $k\ge
c(S)$, then $I_k=\{x\in\N\mid x\ge k\}$. The minimal generating system of the numerical semigroup
$I_k$ is $\{k,k+1,\ldots,2k-1\}$, while the minimal ideal generating system
of $I_k(S)$ may be smaller. In fact, we have
$$
I_k(S)=\{k,k+1,\ldots,k+m-1\}+S^0.
$$
Moreover, when $k\ge c(S)$, we can easily conclude that $\{k,k+1,\ldots,k+m-1\}$ is the minimal ideal
generating system of $I_k$.

\subsection{Rees quotients}\label{subsec:rees_quotients}
Given a numerical semigroup $S$ and an ideal $I_S$ we can form the Rees
quotient
$$\rqns(S,I_S)=S/I_S,$$
which is obtained from $S$ by identifying all elements of $I_S$ to a
distinguished element. This element is the \emph{zero} of the semigroup.

Throughout this paper we will call \emph{Rees quotient numerical semigroup}
(\rqns, for short) to any semigroup that is isomorphic to a Rees quotient of a
numerical semigroup by an ideal.

The \emph{zero} of a \rqns\ will be denoted by $\infty$.  The non-zero
elements of a \rqns\ are said to be \emph{finite} and the set of finite
elements is said to be the \emph{finite part} of the \rqns.

We denote by $\pv{RQNS}$ the class of all Rees quotient numerical semigroups.
\subsection{Nilpotency}\label{subsec:nilpotency}
Let $S$ be a finite semigroup. Let $n$ be a positive integer. As usual, we
write $S^n$ for $\{s_1\cdots s_n\mid s_1,\ldots,s_n\in S\}$.  We say that $S$ is
\emph{nilpotent} if $|S^n|=1$, for some positive integer $n$. The least such
$n$ is called the \emph{nilpotency class} of the nilpotent semigroup $S$.  As
usual when dealing with finite semigroups, $x^{\omega}$ denotes the
idempotent that is a power of $x$, for any element $x$ of the
semigroup. The following fact is well known.
\begin{fact}
  Let $S$ be a semigroup with zero. The following conditions are equivalent:
  \begin{enumerate}[i)]
  \item $S$ is nilpotent;
  \item $S$ satisfies an equation of the form $x_1\cdots x_n=0$, for some
    $n\in\N$;
  \item $S$ satisfies an equation of the form $x^n=0$, for some $n\in\N$;
  \item $S$ satisfies the pseudoequation $x^{\omega}=0$.
  \end{enumerate}
\end{fact}

Notice that, being $z$ a (pseudo)word, the expression $z=0$ is just an abbreviation for the (pseudo)equations $z y=y z=z$, with $y$ a fixed variable not occurring in $z$.

The following is immediate.

\begin{remark}\label{rem:qns-finite-commutative}
  A \rqns\ is a nilpotent finite commutative semigroup. 
 Furthermore, if $S$ is a numerical semigroup with multiplicity $m$ and conductor
  $c$, then the nilpotency class of $\rqns(S,I_{c})$ is the least positive integer
  $k$ such that $km\ge c$.
\end{remark}


\subsection{Some classes of finite semigroups}\label{subsec:classes}
A \emph{pseudovariety} of semigroups is a class of finite semigroups closed
under the formation of finite direct products, subsemigroups and homomorphic
images.

Among the classes of finite semigroups considered in this paper are the
classes $\pv{N}$ of nilpotent semigroups and $\pv{Com}$ of commutative
semigroups.  Remark~\ref{rem:qns-finite-commutative} tells us that
$\pv{RQNS}\subseteq\pv{N}\cap\pv{Com}.$

Other classes also considered here are subclasses of $\pv{RQNS}$, obtained by
considering ideals determined by cutting points. Specifically, we
consider $$\pv{CQNS}=\{S/I_k(S)\mid S \mbox{ is a numerical semigroup and
}k\in\NN\}\quad\mbox{ and } \quad\pv{C\NN}=\{\NN/I_k\mid k\in\NN\}.$$

A semigroup in $\pv{CQNS}$ will be referred as a \emph{quotient determined by
  the numerical semigroup $S$ and the cutting point $k$}.

As follows from the definitions we have
$\pv{C\NN}\subseteq\pv{CQNS}\subseteq\pv{RQNS}$ and therefore we have the
following chain:
\begin{equation}\label{eq:chain}
  \pv{C\NN}\subseteq\pv{CQNS}\subseteq\pv{RQNS}\subseteq\pv{N}\cap\pv{Com}.
\end{equation}

All the inclusions are strict, as we shall see in
Section~\ref{sec:examples_remarks}.  In Section~\ref{sec:pseudovarieties} we
prove that the pseudovariety generated by the smallest of these classes,
$\pv{C\NN}$, is $\pv{N}\cap\pv{Com}$, getting this way a nice set of
generators for the pseudovariety $\pv{N}\cap\pv{Com}$.  Before, in
Section~\ref{sec:presentations}, we give presentations for semigroups of
$\pv{RQNS}$.  Then, we devote Section~\ref{sec:open_problems} to open
problems. Among these problems are possible definitions corresponding to the
notable elements. 

\section{Examples and simple remarks}\label{sec:examples_remarks}
When we want to represent a set of integers that contains all the integers
from, say, $a_n$ on, we use the notation $\{a_1,\ldots,a_n\all\}$ instead of
the more common $\{a_1,\ldots,a_n,\to\}$ mainly because it is more compact
and makes the table in Subsection~\ref{subsec:table} more readable.
Singleton sets will be usually represented by the single element they
contain. The exceptions occur when we want to stress out which are the
elements in an equivalence class.

The following just intends to give a first example that helps gaining
some intuition. We observe that the finite semigroup involved appears as
quotient of two different numerical semigroups.
\begin{example}\label{example:iso_quotients}
  Consider the following numerical semigroups, ideals and Rees quotients:
  \begin{enumerate}
  \item $S=\langle 2,5\rangle$, $I_S=\{6,7\}+S=\{6\all\}$ and
    $Q_1=\rqns(S,I_S)= \{\{2\},\{4\},\{5\},\infty\}$;
  \item $T=\langle 3,5\rangle$, $I_T=\{8,9,10\}+S=\{8\all\}$ and
    $Q_2=\rqns(T,I_T)= \{\{3\},\{5\},\{6\},\infty\}$.
  \end{enumerate}
  It is straightforward to observe that the function $\varphi:Q_1\to Q_2$
  defined by $\varphi(\{2\})=\{3\}, \varphi(\{4\})=\{6\},
  \varphi(\{5\})=\{5\}$ and $\varphi(\infty)=\infty$ is an isomorphism.
\end{example}

When a \rqns\ is defined by a numerical semigroup and a cutting point, there
exists a numerical semigroup such that the cutting point is the conductor of
the semigroup, as stated in the following remark.
\begin{remark}
  Let $S$ be a numerical semigroup and $k$ a positive integer. There exists a
  numerical semigroup $T$ such that $\rqns(S,I_k)=\rqns(T,T_{c(T)})$.
\end{remark}
\begin{proof}
  If $k\le c(S)$, take $T=S\cup\{n\in \NN\mid n \ge k\}$. If $k>c(S)$, take
  $T=2S\cup\{n\in \NN\mid n\ge 2k\}$.
\end{proof}

\begin{example}
  Let $S=\langle 3,5\rangle$ and take $k=10$ as cutting point. We get
  $S/I_k=\{3,5,6,8,9,\infty\}$.  One may take
  $T=\{6,10,12,16,18,20\all\}$. As $c(T)=20$, we get
  $T/I_{c(T)}=\{6,10,12,16,18,\infty\}$, which is isomorphic to $S/I_k$.
\end{example}

Notice that if $a,b,c$ are finite elements of a \rqns\ such that if $a+b=a+c$
is finite, then $b=c$. We state this property in the following proposition.
\begin{proposition}\label{prop:cancelation}
  The cancellation law holds in the finite part of a \rqns.
\end{proposition}
This proposition is the main ingredient used in the examples of the following
subsection which show that there exist quotients of numerical semigroups
that are not Rees quotients by ideals, and that the direct product of Rees
quotient numerical semigroups is not necessarily a \rqns, respectively.
\subsection{$\pv{RQNS}\ne\pv{N}\cap\pv{Com}$}\label{subsec:RQNSnonPV}
The following examples show that the class $\pv{RQNS}$ is not a
pseudovariety. In fact, it fails to be closed under the formation of
homomorphic images (Example~\ref{example:quotient_non_Rees}) and under the
formation of direct products (Example~\ref{example:direct_product_non_Rees}).
\begin{example}\label{example:quotient_non_Rees}
  Let $S=\rqns(\langle4,5\rangle,\{12\all\})=\{4,5,8,9,10,\infty\}$ and let
  $\theta$ be the congruence of $S$ generated by the pair $(9,10)$.  Then, it
  is easy to verify that
  $S/\theta=\{\{4\},\{5\},\{8\},\{9,10\},\{12\all\}\}$ ($\{12\all\}$ is a
  zero). We have that $\{4\}+\{5\}=\{9,10\}=\{5\}+\{5\}$, whence $S/\theta$
  does not satisfy the ``cancellation law''
  (Proposition~\ref{prop:cancelation}) and therefore is not the Rees quotient
  of a numerical semigroup.\\
  \begin{center}
    \begin{tikzpicture}[fill=gray!30]
      \path (1,4) node(a) [rectangle,draw] {$\,4\,$} (3,4) node(b)
      [rectangle,draw] {$\,5\,$} (0,2) node(aa) [rectangle,draw] {$\,8\,$}
      (2,2) node(ab) [rectangle,draw,fill] {$\,9\,$} (4,2) node(bb)
      [rectangle,draw,fill] {$\,10\,$} (2,0) node(z) [rectangle,draw]
      {$\,\infty^*\,$}; \path[-,black] (a) edge (ab); \path[-,black] (a) edge
      (aa); \path[-,black] (b) edge (ab); \path[-,black] (b) edge (bb);
      \path[-,black] (b) edge (ab); \path[-,black] (ab) edge (z);
      \path[-,black] (aa) edge (z); \path[-,black] (bb) edge (z);
    \end{tikzpicture}
    \hspace*{3cm}
    \begin{tikzpicture}[fill=green!30]
      \path (0,4) node(a) [rectangle,draw,fill] {$\{4\}$} (2,4) node(b)
      [rectangle,draw,fill] {$\{5\}$} (0,2) node(c) [rectangle,draw]
      {$\{8\}$} (2,2) node(ab) [rectangle,draw,fill] {$\{9,10\}$} (1,0)
      node(z) [rectangle,draw] {$\infty^*$}; \path[-,black] (a) edge (ab);
      \path[-,black] (a) edge (c); \path[-,black] (b) edge (ab);
      \path[-,black] (b) edge (ab); \path[-,black] (ab) edge (z);
      \path[-,black] (c) edge (z);
    \end{tikzpicture}
  \end{center}
\end{example}

\begin{example}\label{example:direct_product_non_Rees}
  Let $S=\rqns(\langle2,5\rangle,\{4\all\})=\{2,\infty\}$ and
  $T=\rqns(\langle2,7\rangle,\{6\all\})=\{2,4,\infty\}$. Then, the direct
  product
$$
S\times
T=\{(2,2),(\infty,2),(2,4),(\infty,4),(2,\infty),(\infty,\infty)=\infty\}
$$ 
is not isomorphic to a \rqns. In fact, we have
$(2,2)+(2,2)=(2,2)+(\infty,2)=(\infty,2)+(\infty,2)=(\infty,4)$, which can
not happen in a \rqns, due (again) to the ``cancellation law''
(Proposition~\ref{prop:cancelation}).

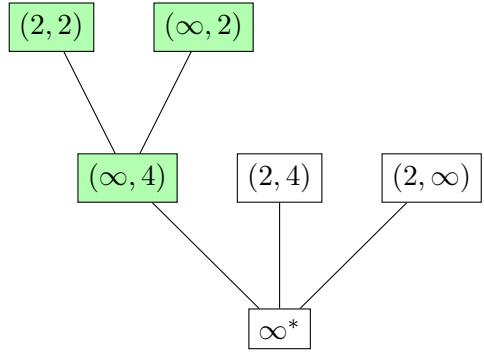
\begin{center}
  \begin{tikzpicture}[fill=green!30]
    \path (0,4) node(a) [rectangle,draw,fill] {$(2,2)$} (2,4) node(b)
    [rectangle,draw,fill] {$(\infty,2)$} (1,2) node(ab) [rectangle,draw,fill]
    {$(\infty,4)$} (3,2) node(c) [rectangle,draw] {$(2,4)$} (5,2) node(d)
    [rectangle,draw] {$(2,\infty)$} (3,0) node(z) [rectangle,draw]
    {$\infty^*$}; \path[-,black] (a) edge (ab); \path[-,black] (b) edge (ab);
    \path[-,black] (ab) edge (z); \path[-,black] (c) edge (z); \path[-,black]
    (d) edge (z);
  \end{tikzpicture}
\end{center}
\end{example}
\begin{corollary}\label{cor:RQNS_non_pseudovariety}
  The class \pv{RNQS} is not closed under neither the formation of
  homomorphic images nor the formation of direct products.
\end{corollary}
\subsection{$\pv{CQNS}\ne\pv{RQNS}$}\label{subsec:CQNSneRQNS}

The following lemma, whose statement is inspired on Distler's
work~\cite{Distler2012tr-Finite} on the classification of finite nilpotent
semigroups, will allow us to prove that not every \rqns\ is determined by a
numerical semigroup and a cutting point.  Observe that for an integer $i$,
and a finite nilpotent semigroup $S$, the set $S^i$ consists of the elements
of $S$ that can be written as the product of at least $i$ minimal
generators. Therefore, $S^i\setminus S^{i+1}$ consists of the elements that
can be written as the product of exactly $i$ minimal generators.
\begin{lemma}\label{lemma:quotient_not_determined_by_cutting_point}
  Let $S$ be a commutative nilpotent semigroup
  minimally generated by $\{a,b\}$, such that $a^3=b^3=0$, $|S^2\setminus
  S^3|=2$ and $|S^3\setminus S^4|=1$. Then $S$ is not determined by a
  numerical semigroup and a cutting point.
\end{lemma}
Prior to the proof we give an example
of a semigroup
fulfilling the conditions of the proposition. 

\begin{example}\label{ex:quotient_not_determined_by_cutting_point}
  Let us consider the numerical semigroup $N=\langle
  3,5\rangle=\{3,5,6,8\all\}$ and its ideal
  $I_N=\{6\}+N^0=\{6,9,11,12,14\all\}$. The quotient
  $S=N/I_N=\{3,5,8,10,13,\infty\}$ has the following multiplication table

$$\begin{array}{c||c|c|c|c|c|c|}
  \oplus &3&5&8&10&13&\infty\\
  \hline\hline
  3& \infty& 8& \infty& 13& \infty& \infty 
  \\ \hline
  5& 8& 10& 13& \infty& \infty& \infty 
  \\ \hline
  8& \infty& 13& \infty& \infty& \infty& \infty 
  \\ \hline
  10& 13& \infty& \infty& \infty& \infty& \infty 
  \\ \hline 
  13& \infty& \infty& \infty& \infty& \infty& \infty 
  \\ \hline
  \infty& \infty& \infty& \infty& \infty& \infty& \infty 
\end{array}
$$
and the following structure in Green's $\mathcal{D}$-classes\\
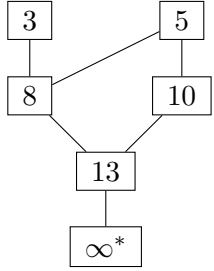
\begin{center}
  \begin{tikzpicture}[fill=gray!30]
    \path (0,3) node(a) [rectangle,draw] {$\,3\,$} (2,3) node(b)
    [rectangle,draw] {$\,5\,$} (0,2) node(ab) [rectangle,draw] {$\,8\,$}
    (2,2) node(bb) [rectangle,draw] {$\,10\,$} (1,1) node(abb)
    [rectangle,draw] {$\,13\,$} (1,0) node(z) [rectangle,draw]
    {$\,\infty^*\,$}; \path[-,black] (a) edge (ab); \path[-,black] (b) edge
    (ab); \path[-,black] (b) edge (bb); \path[-,black] (ab) edge (abb);
    \path[-,black] (bb) edge (abb); \path[-,black] (abb) edge (z);
  \end{tikzpicture}
\end{center}
\end{example}
Observe that the
constructed semigroup is a \rqns, which proves the following Corollary:
\begin{corollary}\label{cor:rqns_non_cqns}
  A \rqns\ is not necessarily determined by a numerical semigroup and a
  cutting point.
\end{corollary}
Next we prove Lemma~\ref{lemma:quotient_not_determined_by_cutting_point}.
\begin{proof}
  Suppose that there exist a numerical semigroup $M$, a positive integer
  $k$ and an isomorphism $f:M/I_k\to S$. 
  Let
  $p:M\to M/I_k$ be the canonical projection, i.e. 
$$
p(m) = \left\{ \begin{array}{ll}
    \{m\}~,& m < k\\
    \infty~, &m\ge k,
  \end{array}\right.
$$
for all $m \in M$, and take the surjective homomorphism $g=f\circ p:M\to S$.  Clearly, being $\{n_1,n_2,\ldots, n_t\}$, with
$n_1<n_2<\cdots< n_t$, the minimal generators of $M$, we have $t\ge 2$
and $\{g(n_1),g(n_2)\}=\{a,b\}$.  Without loss of generality, we may admit 
that $g(n_1)=a$.  Now, as $S$ is commutative, we have  $S^2\setminus S^3\subseteq
\{a^2=g(2n_1),ab=g(n_1+n_2),b^2=g(2n_2)\}$. Moreover, since $2n_1<n_1+n_2<2n_2$ and
$|S^2\setminus S^3|=2$, we deduce that $S^2\setminus S^3=\{a^2, ab\}$.
Thus, again by the commutativity of $S$, we obtain  $S^3\setminus S^4\subseteq \{a^3=g(3n_1),a^2b=g(2n_1+n_2),ab^2=g(n_1+2n_1)\}$.  Now, as $a^3=0$, we get $k\le
3n_1<2n_1+n_2<n_1+2n_2$, whence $a^2b=ab^2=0$
and so $S^3\setminus S^4=\emptyset$, a
contradiction.
\end{proof}

\section{Presentations}\label{sec:presentations}
We refer the reader to~\cite{Ruskuk2000phd-Semigroup} for presentations of
semigroups in general. Presentations of numerical semigroups may be found in
our general reference~\cite{RosalesGarcia2009Book-Numerical}.

Our aim is to find a nice presentation for a \rqns.  We use the presentation
obtained to give a practical method to compute the automorphism group of a
\rqns.

\subsection{Rees quotient semigroup presentations}
We start with some generalities.

Let $S$ be a semigroup without identity. 
As usual, we denote by $S^1$ the semigroup $S$ with an adjoined identity.   
An element $x\in S$ is called \textit{indecomposable in
$S$} if there are no elements $a,b\in S$ such that $x=ab$.
Let $I$ be a proper ideal of $S$. Then, clearly, 
an element $x\in S\setminus I$ is indecomposable in $S$ if and only if
$[x]_I=\{x\}$ is indecomposable in $S/I$.
On the other hand, if $X$ is a set of generators of $S$ then 
$\{[x]_I=\{x\}\mid x\in X\setminus I\}\cup\{I\}$ is a set of generators of $S/I$. 
Notice that, if $S\setminus I$ is not a subsemigroup of $S$, 
then $\{[x]_I=\{x\}\mid x\in X\setminus I\}$ generates $S/ I$ 
(thus $I$ does not need to be added as a generator).

Given a set $X$, we denote by $X^+$ the free semigroup on $X$.

Let $\langle X\mid R\rangle$ be a presentation of $S$. 
For each $y\in I$, denote by $w_y$ a (fixed) word of $X^+$
representing the element $y$. Let $Y$ be an \textit{ideal generating system} of
$I$, i.e. a non-empty subset $Y$ of $I$ such that $I=S^1YS^1$ (notice that any set of generators of $I$ is also an ideal generating system of $I$).  
Then, it is a routine matter to show: 

\begin{lemma}\label{prezero1}
The semigroup $S/I$ is defined by presentation with \textit{zero} 
$\langle X\mid R,~ w_y=0~(y\in Y)\rangle$. 
\end{lemma} 

Recall that, in a presentation (of semigroups) $\langle A\mid \mathfrak{R}\rangle$ with zero, the {\it free semigroup with zero} $A^+_0$ (i.e. the free semigroup $A^+$ with a zero adjoined) plays the same role as $A^+$ in a usual presentation (of semigroups).   

Next, let $\langle X\mid R\rangle$ be a presentation of $S$ and let $Y$ be any subset of $I$ (not necessarily an ideal generating system) such that 
$$
\langle X\mid R,~ w_y=0~(y\in Y)\rangle
$$
is a presentation (with \textit{zero}) of $S/I$. We define from $R$ the following sets of relations on $(X\setminus I)^+_0$: 
$$
R_1=\{(u,v)\mid \mbox{$(u,v)\in R$ and $u,v\in (X\setminus I)^+$}\}, 
$$
$$
R_2=\{(u,0)\mid \mbox{$(u,v)\in R$ or $(v,u)\in R$, 
with $u\in (X\setminus I)^+$ and $v\in X^+\setminus (X\setminus I)^+$}\}
$$
and 
\begin{equation} \label{rels}
R'=R_1\cup R_2. 
\end{equation}
Let $Y'=\{y\in Y\mid w_y\in (X\setminus I)^+\}$ (in
addition, we may assume that, 
for each element $y\in I\cap X$, we have taken $w_y$ as the word $y$ and so, 
in this case, $Y'\cap X =\emptyset$). Under these conditions, we have: 
\begin{lemma}\label{prezero2}
The semigroup $S/I$ is defined by presentation 
$\langle X\setminus I\mid R',~ w_y=0~(y\in Y')\rangle$. 
\end{lemma} 
\begin{proof}
First, observe that, clearly, all the relations of 
$\langle X\setminus I\mid R',~ w_y=0~(y\in Y')\rangle$ are satisfied by $S/I$. 
Therefore, it remains to prove that all the equalities, 
between words of $(X\setminus I)^+$, satisfied by $S/I$ are consequences of 
$R'\cup\{w_y=0\mid y\in Y'\}$. 

Let $w,w'\in (X\setminus I)^+$ be such that the equality $w=w'$ is satisfied
by $S/I$.  Then $w$ and $w'$ both represent elements of $I$ or $w=w'$ is
satisfied by $S$ (and, in this case, we may suppose that $w$ and $w'$
represent an element of $S\setminus I$).

If this last case occurs, since $S$ is presented by $\langle X\mid R\rangle$,
then there exists a finite sequence of elementary $R$-transitions $w\to
w_1\to\cdots\to w_{n-1}\to w'$ over $X^+$.  As each word of this sequence
represents the same element of $S\setminus I$, then all the letters involved
must belong to $X\setminus I$. Hence $w\to w_1\to\cdots\to w_{n-1}\to w'$ is
also a finite sequence of elementary $R_1$-transitions over $(X\setminus
I)^+$ and so, in particular, $w=w'$ is a consequence of $R'\cup\{w_y=0\mid
y\in Y'\}$.

Next, suppose that $w$ represents an element of $I$.  Since $\langle X\mid
R,~ w_y=0~(y\in Y)\rangle$ is a presentation of $S/I$ and the equality $w=0$
is satisfied by $S/I$, then there exist a finite sequence of elementary
$R$-transitions $w\to w_1\to\cdots\to w_{n-1}$ over $X^+$ and an elementary
$\{w_y=0\mid y\in Y\}$-transition $w_{n-1}\to 0$ over $X^+$.  Now, we
consider two possibilities.  First, if $w_1,\ldots,w_{n-1}\in (X\setminus
I)^+$, then $w\to w_1\to\cdots\to w_{n-1}$ is also a finite sequence of
elementary $R_1$-transitions over $(X\setminus I)^+$ (notice that, we have
taken $w\in (X\setminus I)^+$) and $w_{n-1}\to 0$ is an elementary
$\{w_y=0\mid y\in Y'\}$-transition over $(X\setminus I)^+$, whence $w=0$ is a
consequence of $R'\cup\{w_y=0\mid y\in Y'\}$.  Secondly, we suppose that
there exists an index $i\in\{1,\ldots,n-1\}$ such that $w_i\in
X^+\setminus(X\setminus I)^+$ and take the smallest of such indexes.  Thus
$w\to w_1\to\cdots\to w_{i-1}$ (where $w_{i-1}$ denotes the word $w$, for
$i=1$) is a finite sequence of elementary $R_1$-transitions over $(X\setminus
I)^+$ and $w_{i-1}\to w_i$ is an elementary $\{(u,v)\}$-transition, for some
$u\in (X\setminus I)^+$ and $v\in X^+\setminus (X\setminus I)^+$ such that
$(u,v)\in R$ or $(v,u)\in R$. Hence $w\to w_1\to\cdots\to w_{i-1}\to 0$ is a
finite sequence of elementary $R_1\cup R_2$-transitions over $(X\setminus
I)^+$ and so again we conclude that $w=0$ is a consequence of
$R'\cup\{w_y=0\mid y\in Y'\}$.

Therefore, if $w$ and $w'$ both represent elements of $I$, from the last
paragraph we deduce that $w=w'$ is a consequence of $R'\cup\{w_y=0\mid y\in
Y'\}$, as required.
\end{proof}
\subsection{Presentations of \rqns's}
Now, let $S$ be a numerical semigroup and let $I$ be a proper ideal of
$S$. Take an ideal generating system $G$ of $I$ (whence $I=G+S^0$) and a
presentation $\langle X\mid R\rangle$ of $S$. Then, as a particular case of
Lemma \ref{prezero1}, we have:

\begin{corollary}\label{prop:qnspre}
  The \rqns~$S/I$ is defined by the presentation $\langle X\mid R,~
  w_g=0~(g\in G)\rangle$.
\end{corollary}

\begin{example}
Let $m=m(S)$ and let $k\ge c(S)$ be an integer. Recall that
$\{k,k+1,\ldots,k+m-1\}$ is the minimal ideal generating system of the ideal $I_k$ of $S$. 
Thus 
$$
\langle X\mid R,~w_k=w_{k+1}=\cdots=w_{k+m-1}=0\rangle
$$
is a presentation of $S/I_k$.
\end{example}

Next, we suppose that $X$ is the minimal generating system of $S$.  Then (the
set of classes of the elements of) $X\setminus I$ is a set of generators of
$S/I$ (since $S\setminus I$ is never a subsemigroup of $S$). Moreover, as
$X\setminus I\subseteq S/I$ is also a set of indecomposable elements of
$S/I$, then $X\setminus I$ is a minimum set (for set inclusion) of generators
of $S\setminus I$. We call to $X\setminus I$ the \textit{minimal generating
  system} of the \rqns~$S/I$.

In particular, if the ideal $I$ does not contain any element of the minimal
generating system $X$ of $S$, then $X$ is also the minimal generating system
of $S/I$. In this case, the presentation of $S/I$ given by 
Corollary~\ref{prop:qnspre} (considering there $X$ as being the minimal
generating system of $S$) is already over its minimal generating system.

In what follows, we aim to determine a presentation over the minimal
generating system of any \rqns.

Let $G$ be the minimal ideal generating system of $I$. Recall that, 
for each $g\in G$, we denote by $w_g$ a fixed word of $X^+$ 
representing the element $g$. Then, we have:

\begin{lemma}\label{lem:qnspremgs}
For any $g\in G$, $w_g\in (X\setminus I)^+$ if and only if $g\in G\setminus X$.
\end{lemma}
\begin{proof}
If $g\in X$ then $g$ is indecomposable, whence $w_g$ should be the word $g$ 
(despite, in any case, we could have chosen $w_g$ equal $g$) and so 
$w_g\not\in (X\setminus I)^+$. 

Conversely, suppose that $w_g\not\in (X\setminus I)^+$. Then
$g=x_1+\cdots+x_n$, with $x_1,\ldots,x_n\in X$ ($n\ge1$) and $x_i\in I$, for
some $1\le i\le n$.  Then, as $x_i$ is indecomposable and $x_i\in G+S^0$, we
must have $x_i\in G$ (since $x_i=x_i+0$ is the unique decomposition
permitted).  Now, if $g\neq x_i$, then $I=G+S^0\subseteq
(G\setminus\{g\})+S^0\subseteq I$ (observe that, given $x\in S^0$, we have
$g+x=x_i+(x_1+\cdots+x_{i-1}+x_{i+1}+\cdots+x_n+x)\in(G\setminus\{g\})+S^0$),
which contradicts the minimality of $G$. Thus $g=x_i\in X$ and so the lemma
is proved.
\end{proof}

Let $R'$ be the set of relations over $(X\setminus I)^+_0$ obtained
from $R$ as defined in general in (\ref{rels}). 
Thus, combining Lemma \ref{prezero2} with Lemma \ref{lem:qnspremgs}, we immediately have:

\begin{theorem}\label{prop:qnspremgs}
  The \rqns~$S/I$ is defined by the presentation $\langle X\setminus I\mid
  R',~ w_g=0~(g\in G\setminus X)\rangle$ over its minimal generating system.
\end{theorem}

\subsection{Isomorphisms of \rqns's}
In what follows, all Rees quotient numerical semigroups considered are
constructed from proper ideals.

Let $S_1$ and $S_2$ be two Rees quotient numerical semigroups with minimal
generating systems $X_1$ and $X_2$, respectively.

\begin{lemma}
  If $\varphi:S_1\longrightarrow S_2$ is a surjective homomorphism, then
  $X_2\subseteq \varphi(X_1)$. In particular $|X_2|\le|X_1|$.
\end{lemma}
\begin{proof}
  Since $\varphi$ is surjective, then $\varphi(X_1)$ generates $S_2$. Thus
  $X_2\subseteq \varphi(X_1)$, by the minimality of $X_2$. Moreover,
  $|X_2|\le|\varphi(X_1)|\le|X_1|$, as required.
\end{proof}

It follows immediately that:

\begin{proposition}\label{prop:isomgs}
  Let $\varphi:S_1\longrightarrow S_2$ be an isomorphism. Then
  $\varphi(X_1)=X_2$. In particular $|X_1|=|X_2|$.
\end{proposition}
\medskip

Let $S$ and $T$ be (any) two semigroups and let $\langle X\mid R\rangle$ be a
presentation of $S$.

Let $f:X\longrightarrow T$ be a mapping and let $\phi:X^+\longrightarrow T$
be the (unique) homomorphism extending $f$ (regarding $X$ as a set of
letters). If $f$ \textit{satisfies} $R$, i.e. $\phi(u)=\phi(v)$, for all
$(u,v)\in R$, then the mapping $\varphi:S\longrightarrow T$ defined by
$\varphi(s)=\phi(w_s)$, where $w_s$ is any (fixed) word of $X^+$ representing
$s$, for all $s\in S$, is the unique homomorphism extending $f$ (regarding
$X$ as a generating set of $S$). Moreover, $f(X)$ generates $T$ if and only
if $\varphi:S\longrightarrow T$ is a surjective homomorphism.  In particular,
supposing that $S$ is a finite semigroup, if $f(X)$ generates $T$ and
$|S|=|T|$, then $\varphi:S\longrightarrow T$ is an isomorphism.

Conversely, let $\varphi:S\longrightarrow T$ be a homomorphism and let
$f:X\longrightarrow T$ be the restriction of $\varphi$ to $X$. Then, clearly,
$f$ must satisfy $R$.  \medskip

In view of Proposition~\ref{prop:isomgs} and the above observations, we have
the following interesting conclusion regarding isomorphisms of Rees quotient
numerical semigroups.

\begin{theorem}\label{th:qnsiso}
  Let $S_1$ and $S_2$ be two Rees quotient numerical semigroups with minimal
  generating systems $X_1$ and $X_2$, respectively. Let $\langle X_1\mid
  R\rangle$ be a (fixed) presentation of $S_1$.  If $|S_1|=|S_2|$ then the
  isomorphisms from $S_1$ to $S_2$ are precisely the homomorphisms
  $\varphi:S_1\longrightarrow S_2$ extending bijections $f:X_1\longrightarrow
  X_2$ satisfying $R$.
\end{theorem}

For a numerical semigroup $S$ with multiplicity $m$ and embedding dimension
$e$, we may compute a presentation with less than or equal to
$\frac{(2m-e+1)(e-2)}{2}+1$ relations
\cite{RosalesGarcia2009Book-Numerical}. Thus, if $I$ is an ideal of $S$ with
minimal ideal generating system $G$, regarding Theorem~\ref{prop:qnspremgs},
we may compute a presentation (with zero) for the \rqns~$S/I$ over its
minimal generating system with less than or equal to
$\frac{(2m-e+1)(e-2)}{2}+|G|+1$ relations. Therefore, Theorem~\ref{th:qnsiso}
gives us a reasonable practical method to compute all the isomorphisms
between two Rees quotient numerical semigroups (at least for the ones having
\textit{small-size} minimal generating systems) and, in particular, to
compute the automorphism group of a \rqns:
\begin{corollary}
  Let $S$ be a \rqns~with minimal generating system $X$ and let $\langle
  X\mid R\rangle$ be a (fixed) presentation of $S$. Then, the automorphisms
  of $S$ are the endomorphisms of $S$ that extend permutations of $X$
  satisfying $R$.
\end{corollary}

A recent paper by García-García and
Moreno~\cite{GarciaMoreno2012SF-morphisms} treat problems on morphisms of
commutative monoids which at first sight could be thought as similar to the
ones treated in this subsection, but we have not discovered any strong
connection.


\section{Generators of the pseudovariety
  $\pv{N}\cap\pv{Com}$}\label{sec:pseudovarieties}
Recall that $\pv{C\NN}=\{\rqns(\NN,I_k)\mid k\in \NN \}$. We show that this
\textit{small} class, consisting of easily described finite semigroups, forms
a set of generators of the pseudovariety of finite nilpotent and commutative
semigroups.
\begin{theorem}\label{th:pseudovariety_generators}
  The class $\pv{C\NN}$ generates the pseudovariety $\pv{N}\cap\pv{Com}$.
\end{theorem}
\begin{proof}
  Let $\pv{V}$ be the pseudovariety generated by $\pv{C\NN}$. Then, as
  $\pv{V}\subset\pv{Com}$, by \cite[Theorem 6.2.6]{Almeida1995Book-Finite},
  $\pv{V}$ is admits a finite basis of pseudoidentities of the form
$$
\Sigma\cup\{xy=yx,\pi(x)x^\omega=x^\omega\},
$$
where $\pi$ is a pseudoword on one variable such that $\pv{V}\cap\pv{G}$ is
defined by the pseudoidentities $\pi=1$ and $xy=yx$ and $\Sigma$ consists of
pseudoidentities which are valid in $\pv{V}$ of the form
$x_1^{\alpha_1}\cdots x_n^{\alpha_n}=x_1^{\beta_1}\cdots x_n^{\beta_n}$, with
$\alpha_1,\ldots,\alpha_n,\beta_1,\ldots,\beta_n\in\N_0\cup\{\omega\}$ and
variables $x_1,\ldots,x_n$ not necessarily distinct.

As $\pv{N}$ is defined by the pseudoidentity $x^\omega=0$ and
$\pv{V}\subset\pv{N}$, we may replace $\pi(x)x^\omega=x^\omega$ simply by
$x^\omega=0$.

Next, suppose that $\pv{V}$ is strictly contained in $\pv{N}\cap\pv{Com}$.
Hence $\pv{V}$ must satisfy a non-trivial pseudoidentity, which is not
satisfied by $\pv{N}\cap\pv{Com}$, of the form $x_1^{\alpha_1}\cdots
x_n^{\alpha_n}=x_1^{\beta_1}\cdots x_n^{\beta_n}$, with
$\alpha_1,\ldots,\alpha_n,\beta_1,\ldots,\beta_n\in\N_0\cup\{\omega\}$ and
$x_1,\ldots,x_n$ not necessarily distinct. As $\pv{V}$ also satisfies
$x^\omega=0$, it follows that this pseudoidentity must be an identity of the
form $x_1^{\alpha_1}\cdots x_n^{\alpha_n}=x_1^{\beta_1}\cdots x_n^{\beta_n}$,
with $\alpha_1,\ldots,\alpha_n,\beta_1,\ldots,\beta_n\in\N_0$, or of the form
$x_1^{\alpha_1}\cdots x_n^{\alpha_n}=0$, with
$\alpha_1,\ldots,\alpha_n\in\N$. Moreover, in both cases, we may assume that
the variables $x_1,\ldots,x_n$ are distinct.

Let $r$ be a positive integer and consider the semigroup
$Q_r=\rqns(\NN,I_r)$. If $r > \alpha_1+\cdots+\alpha_n$, then we have
$\alpha_1+\cdots+\alpha_n\neq\infty$ in $Q_r$, whence $Q_r$ does not satisfy
the identity $x_1^{\alpha_1}\cdots x_n^{\alpha_n}=0$ and so the taken
pseudoidentity cannot be of this form.  Now, let $r >
\max\{\alpha_1+\cdots+\alpha_{i-1}+2\alpha_i+\alpha_{i+1}+\cdots+\alpha_n\mid
2\le i\le n\}$. Since $Q_r$ must satisfy the identity $x_1^{\alpha_1}\cdots
x_n^{\alpha_n}=x_1^{\beta_1}\cdots x_n^{\beta_n}$, on one hand, we have
\begin{equation}\label{eq:1}
  \infty\ne\alpha_1+\cdots+\alpha_{i-1}+\alpha_i+\alpha_{i+1}+\cdots+\alpha_n = 
  \beta_1+\cdots+\beta_{i-1}+\beta_i+\beta_{i+1}+\cdots+\beta_n
\end{equation}
in $Q_r$ and, on the other hand, for $2\le i\le n$ (by considering $x_j=1$,
for $j\ne i$, and $x_i=2$), we have
\begin{equation}\label{eq:2}
  \infty\ne\alpha_1+\cdots+\alpha_{i-1}+2\alpha_i+\alpha_{i+1}+\cdots+\alpha_n = 
  \beta_1+\cdots+\beta_{i-1}+2\beta_i+\beta_{i+1}+\cdots+\beta_n
\end{equation}
in $Q_r$. Hence, for each $2\le i\le n$, by combining (\ref{eq:1}) with
(\ref{eq:2}), we deduce that $\alpha_i=\beta_i$. Then, by (\ref{eq:1}) it
follows also that $\alpha_1=\beta_1$. Therefore, the identity
$x_1^{\alpha_1}\cdots x_n^{\alpha_n}=x_1^{\beta_1}\cdots x_n^{\beta_n}$ is
trivial, which is a contradiction.

Thus $\pv{V}=\pv{N}\cap\pv{Com}$, as required.
\end{proof}

\section{Open problems}\label{sec:open_problems}
In this section we state some natural questions which we have not yet been
able to answer, opening this way a line of research.

We divide the section into subsections, one of which contains a table
where the numerical semigroups with Frobenius number up to $10$ are
listed. An analysis of this table (and some other computations of the same
kind) leads to several questions and conjectures. We state a decidability
question in the last subsection.
\subsection{Some more terminology}\label{subsec:more_terminology}
Next we grasp some more terminology from our references.

A numerical semigroup is \emph{irreducible} if it cannot be expressed as the
intersection of two numerical semigroups properly containing it.

An irreducible numerical semigroup is either \emph{symmetric}, if its
Frobenius number is odd, or is \emph{pseudo-symmetric}, if its Frobenius
number is even.

The elements of a numerical semigroup that are not greater than the conductor
are called \emph{small elements}. (The name is taken from the
\emph{numericalsgps} \GAP\
package~\cite{DelgadoMoraisGarcia-Sanchez:numericalsgps}, although the
current version of the package always includes the $0$ as a small element of
a semigroup.)

As the set of small elements of a numerical semigroup completely determines
it, we can represent the numerical semigroup through its small elements.
\subsection{A table of semigroups with small Frobenius
  numbers}\label{subsec:table}
Table~\ref{table:FrobLess10} lists the numerical semigroups whose conductor
is not greater than $11$. The use of different colors (or gray tones) will be
explained in the next subsection.

The table has been constructed with the help of the \emph{numericalsgps}
\GAP\ package~\cite{DelgadoMoraisGarcia-Sanchez:numericalsgps}. For instance,
the data in the following example was used to produce the line corresponding
to the semigroups with Frobenius number $8$. 
\begin{example}\label{ex:gap_for_table}
  The following lines correspond to a \GAP\ session where the irreducible
  numerical semigroups with Frobenius number $8$ are determined.  {\small
\begin{verbatim}
gap> LoadPackage("numericalsgps");
true
gap> n := 8;; 
gap> irrn := IrreducibleNumericalSemigroupsWithFrobeniusNumber(n);;
gap> List(irrn,s->SmallElementsOfNumericalSemigroup(s));
[ [ 0, 5, 6, 7, 9 ], [ 0, 3, 6, 7, 9 ] ]
\end{verbatim}
  } \medskip

  \hspace{-\parindent} In the same \GAP\ session, the non irreducibles can be
  determined as follows: {\small
\begin{verbatim}
gap> fn:=NumericalSemigroupsWithFrobeniusNumber(n);;
gap> nirrn:=Filtered(fn,s->not IsIrreducibleNumericalSemigroup(s));;
gap> List(nirrn,s->SmallElementsOfNumericalSemigroup(s));
[ [ 0, 5, 7, 9 ], [ 0, 6, 7, 9 ], [ 0, 7, 9 ], [ 0, 5, 6, 9 ], 
  [ 0, 3, 6, 9 ], [ 0, 5, 9 ], [ 0, 6, 9 ], [ 0, 9 ] ]
\end{verbatim}
  }
\end{example}
The numbers in the first column of Table~\ref{table:FrobLess10} indicate the
Frobenius numbers of the numerical semigroups on the right cells. The second
column contains the irreducible semigroups and the third contains the non
irreducible ones.
\begin{table}[ht]
  \centering
  {\small
  $$\begin{array}{|l|l|l|}
    \hline
    F&
    \mbox{Irreducibles}&
    \mbox{Non irreducibles}\\
    \hline
    1&
    {\{ 2\all \}}&
    \\
    \hline
    2&
    {\{ 3\all \}}&
    \\
    \hline
    3&
    {\{ 2, 4\all \}}&
    {\{ 4\all \}}
    \\
    \hline
    4&
    {\{ 3, 5\all \}}&
    {\{ 5\all \}}
    \\
    \hline
    5&
    { \{ 3, 4, 6\all \}},{ \color{red}\{ 2, 4, 6\all \}}&
    {\{ 3, 6\all \}},{ \{ 4, 6\all \}}, {\{ 6\all \}} 
    \\
    \hline
    6&
    { \{ 4, 5, 7\all \}}&
    {\{ 5, 7\all \}},{ \{ 4, 7\all \}}, {\{ 7\all \}} 
    \\
    \hline
    7&
    { \{ 3, 5, 6, 8\all \}},{ \color{red}\{ 4, 5, 6, 8\all \}},&
    { \{ 4, 5, 8\all \}},{ \{ 5, 6, 8\all \}},{
      \color{red}\{ 3, 6, 8\all \}},{ \{ 5, 8\all \}}, 
    \\
    &{ \color{green}\{ 2, 4, 6, 8\all \}}&
    {\{ 4, 8\all \}}, { \{ 4, 6, 8\all \}},{
      \{ 6, 8\all \}}, {\{ 8\all \}} 
    \\
    \hline
    8&
    { \color{red}\{ 5, 6, 7, 9\all \}},{ \{ 3, 6, 7, 9\all \}}&
    { \{ 5, 7, 9\all \}},{ \{ 6, 7, 9\all \}},{
      \{ 7, 9\all \}},{ \{ 5, 6, 9\all \}},
    \\
    &&
    { \color{red}\{ 3, 6, 9\all \}}, {\{ 5, 9\all \}},
    {\{ 6, 9\all \}}, {\{ 9\all \}} 
    \\
    \hline
    9&
    { \{ 5, 6, 7, 8, 10\all \}},{ \color{red}\{ 4,
      6, 7, 8, 10\all \}},&  
    { \color{red}\{ 5, 6, 7, 10\all \}},{ \color{red}\{ 6, 7, 8, 10\all
      \}},{ \color{red}\{ 5, 7, 8, 10\all \}},  

    \\
    &{ \color{green}\{ 2, 4, 6, 8, 10\all \}}&
    { \{ 4, 7, 8, 10\all \}},{ \color{red}\{ 5, 6, 8, 10\all
      \}},{ \{ 6, 7, 10\all \}}, 
    \\
    &&
    \{{ 5, 7, 10\all \}}, 
    { \{ 5, 6, 10\all \}},{ \{ 7, 8, 10\all \}},
    \\
    &&
    { \{ 4, 6, 8, 10\all \}},{ \{ 6, 8, 10\all \}}, 
    { \{ 5, 8, 10\all \}},
    \\
    &&
    {\color{red}\{ 4, 8, 10\all \}},{ \{ 7, 10\all \}},{
      \{ 6, 10\all \}},  
    \\
    &&
    {\{ 5, 10\all \}}, 
    {\{ 8, 10\all \}}, {\{ 10\all \}}
    \\
    \hline
    10&
    { \color{red}\{ 4, 7, 8, 9, 11\all \}},{ \{ 6,
      7, 8, 9, 11\all \}},&
    { \color{red}\{ 7, 8, 9, 11\all \}},{ \{ 4, 8, 9, 11\all
      \}},{ \color{red}\{ 6, 7, 9, 11\all \}}, 
    \\
    &{ \color{blue}\mathbf{\{ 3, 6, 8, 9, 11\all \}}}&
    { \{ 7, 9, 11\all \}},{ \color{green}\{ 3, 6, 9, 11\all \}},{
      \color{red}\{ 6, 7, 8, 11\all \}}, 
    \\
    &&
    { \{ 4, 7, 8, 11\all \}},{ \{ 6, 7, 11\all
      \}},{ \color{red}\{ 6, 8, 9, 11\all \}}, 
    \\
    &&
    { \{ 8, 9, 11\all \}},{ \{ 6, 9, 11\all \}},{
      \{ 9, 11\all \}},  
    \\
    &&
    { \{ 7, 8, 11\all \}},{ \{ 6, 8, 11\all \}},{
      \color{red}\{ 4, 8, 11\all \}},  
    \\
    &&
    { \{ 7, 11\all \}},{ \{ 6, 11\all \}},{
      \{ 8, 11\all \}}, {\{ 11\all \}} 
    \\
    \hline
  \end{array}
$$
}
\caption{Numerical semigroups with conductor up to $11$}
\label{table:FrobLess10}
\end{table}
\subsection{Analysis of the table}\label{subsec:table_analysis}
Let $S$ be a numerical semigroup. We observe that the size of the Rees
quotient $\rqns(S,I_{c(S)})$ is precisely the number of small elements of
$S$.

The set of small elements can also be used to represent the \rqns\ where the
conductor has been taken as cutting point.  Notice that, contrary to what
happens when using this form of representing a numerical semigroup, several
sets of integers can represent the same quotient semigroup (since we do not
make any distinction between isomorphic objects).

In this way, Table~\ref{table:FrobLess10} represents also the Rees quotient
numerical semigroups obtained from the numerical semigroups with Frobenius
number less than $11$ through the use of the conductor as cutting point.
There are many repetitions, since many of the numerical semigroups in the
table lead to isomorphic quotients. Quotients of the same size (which, as
observed, are obtained from numerical semigroups with the same number of
small elements) that are isomorphic are represented using the same color.

The reader can check without any difficulty that the isomorphism classes of
the quotient semigroups in the table are completely determined by the size
and the nilpotency class of the semigroups. In fact, in all cases of
semigroups of the same size and nilpotency class, it is straightforward to
construct an isomorphism.  (We observe that these arguments would need to be
greatly refined in order to deal with numerical semigroups with larger
Frobenius numbers.)

The following seems to be true, but we still have no proof:
\begin{conjecture}\label{conj:diff_symmetric_diff_quotients}
  Different symmetric numerical semigroups correspond to non-isomorphic
  quotient numerical semigroups.
\end{conjecture}

A quick look at the first part of the table could lead to ask whether one
could obtain all the Rees quotient numerical semigroups determined by a
cutting point by using the symmetric numerical semigroups and the conductors
as cutting points. The last part of the table shows that this is not the
case, since the Rees quotient of the pseudo-symmetric numerical semigroup
$\{3, 6, 8, 9, 11\all \}$ obtained by cutting through its conductor is not isomorphic
to any Rees quotient of a symmetric numerical semigroup with the same number
of small elements cutting by its conductor. (Notice that the nilpotency class
of $\{3, 6, 8, 9, \infty \}$ is 4, while the nilpotency classes of the
quotients obtained from symmetric numerical semigroups with Frobenius number
$9$ are $2$, $3$ or $5$.)

One could now make the same question replacing symmetric by irreducible,
although the unicity would in this case be out of question. We state the
question precisely:
\begin{question}
  Let $N$ be a numerical semigroup and let $k$ be a cutting point. Does there
  exist an irreducible numerical semigroup $M$ such that
  $\rqns(N/I_k)\simeq\rqns(M,I_{c(S)})$?
\end{question}
The answer is ``No''! Consider the semigroup $\langle 4,11,13,18 \rangle$ and
take the conductor, $15$, as cutting point. The quotient is a nilpotent
semigroup of size $6$ which has $3$ minimal generators and nilpotency class $4$,
as the following \GAP\ session may help to confirm.  {\small
\begin{verbatim}
gap> N := NumericalSemigroup(4,11,13,18);;
gap> SmallElementsOfNumericalSemigroup(N);
[ 0, 4, 8, 11, 12, 13, 15 ]
\end{verbatim}
}

From basic results on irreducible numerical semigroups
(see~\cite{RosalesGarcia2009Book-Numerical}) it follows that an irreducible
numerical semigroup with $6$ small elements must have Frobenius number $11$
or $12$. These can be computed as follows:
\begin{example}\label{ex:nilpotent_with_frob_11_12}
{\small
\begin{verbatim}
gap> n := 11;;
gap> irrn := IrreducibleNumericalSemigroupsWithFrobeniusNumber(n);;
gap> List(irrn,s->SmallElementsOfNumericalSemigroup(s));
[ [ 0, 5, 7, 8, 9, 10, 12 ], [ 0, 4, 5, 8, 9, 10, 12 ], 
  [ 0, 3, 6, 7, 9, 10, 12 ], [ 0, 6, 7, 8, 9, 10, 12 ], 
  [ 0, 2, 4, 6, 8, 10, 12 ], [ 0, 4, 6, 8, 9, 10, 12 ] ]
gap> n := 12;;
gap> irrn := IrreducibleNumericalSemigroupsWithFrobeniusNumber(n);;
gap> List(irrn,s->SmallElementsOfNumericalSemigroup(s));
[ [ 0, 7, 8, 9, 10, 11, 13 ], [ 0, 5, 8, 9, 10, 11, 13 ] ]
\end{verbatim}
}
\end{example}

The only one with nilpotency class $4$ is obtained from $\{3, 6, 7, 9, 10, 12\all \}$, but
it only has $2$ minimal generators, therefore is not isomorphic to $\{4, 8,
11, 12, 13, \infty \}$.

\subsection{Notable elements in quotients}\label{subsec:notable_in_quotients}
It would be interesting to be able to find a reasonable correspondence
between the definitions of notable elements in numerical semigroups given in
Subsection~\ref{subsec:notable} and
similar notions to be defined in Rees quotient numerical semigroups. For
instance, one should be able to give a reasonable definition for the Frobenius
number of a \rqns.  At some point of our research we had the hope that one
could do this via some irreducible numerical semigroup in the same
isomorphism class, but the examples in the previous subsection show that
something different must be tried.

\subsection{Decidability}\label{subsec:decidability}
\begin{question}\label{question:decidable_strong}
  Let $Q$ be a finite commutative nilpotent semigroup. Give an effective way
  to construct a numerical semigroup $S$ and an ideal $I_S$ such that $Q$ and
  $S/I_S$ are isomorphic, if such numerical semigroup and ideal exist.
\end{question}
As usual, we say that a class of finite semigroups is decidable if there is
an algorithm that having as input a finite semigroup, it outputs whether the
semigroup belongs to the class given.  The following question, to which
Section~\ref{sec:pseudovarieties} is related and which would have a positive
answer if Question~\ref{question:decidable_strong} had a positive answer,
appears naturally:
\begin{question}\label{quest:decidable}
  Is $\pv{RQNS}$ decidable?
\end{question}
\section{Acknowledgments}
We would like to thank Pedro García-Sánchez for encouragement and helpful
comments related to this work. 



\begin{thebibliography}{1}

\bibitem{Almeida1995Book-Finite} J.~Almeida.  \newblock {\em Finite
    Semigroups and Universal Algebra}.  \newblock World Scientific,
  Singapore, 1995.

\bibitem{BarucciDobbsFontana1997Book-Maximality} V.~Barucci, D.~E. Dobbs, and
  M.~Fontana.  \newblock {\em Maximality properties in numerical semigroups
    and applications to one-dimensional analytically irreducible local
    domains}.  \newblock Number 598 in Memoirs of the American Mathematical
  Society. American Mathematical Society, 1997.

\bibitem{DelgadoMoraisGarcia-Sanchez:numericalsgps} M.~Delgado, J.~Morais,
  and P.~Garc{\'\i}a-S{\'a}nchez.  \newblock \emph{Numericalsgps -- a
    \textsf{GAP} package on numerical semigroups}, December 2011.  \newblock
  Version number 0.971. \newblock Available via
  \texttt{http://www.gap-system.org/}.

\bibitem{Distler2012tr-Finite} A.~Distler.  \newblock Finite nilpotent
  semigroups of small coclass.  \newblock Communications in Algebra, to
  appear.  \newblock Preprint available in arXiv:1205.2817v1.

\bibitem{GarciaMoreno2012SF-morphisms} J. I. García García and M. A. Moreno
  Frías. \newblock On morphisms of commutative monoids. \newblock Semigroup Forum, 2012, 84, 333-341

\bibitem{RosalesGarcia2009Book-Numerical} J.~Rosales and
  P.~Garc{\'\i}a~S{\'a}nchez.  \newblock {\em Numerical Semigroups}.
  \newblock Springer, 2009.

\bibitem{Ruskuk2000phd-Semigroup} N. Ruskuk.  \newblock Semigroup
  Presentations.  \newblock P.H.D. thesis.  \newblock University of
  St. Andrews, 2000.

\bibitem{GAP4} The GAP~Group, \emph{GAP -- Groups, Algorithms, and
    Programming, Version 4.5.5}; 2012, \newblock Available via
  \texttt{http://www.gap-system.org/}.
\end{thebibliography}
\end{document}